\documentclass[12pt]{article}
\usepackage{latexsym}
\usepackage{amsmath}
\usepackage{amssymb}
\usepackage{amscd}
\usepackage{array}
\usepackage{comment}
\usepackage{amsthm}
\usepackage{amsfonts}
\usepackage{enumerate}
\usepackage{hyperref}
\usepackage{stackrel}
\usepackage{pdflscape}
\newtheorem{theorem}{Theorem}[]

\newtheorem{lemma}[theorem]{Lemma}
\newtheorem{corollary}[theorem]{Corollary}

\theoremstyle{definition}

\newcommand{\Z}{\mathbf Z}

\newcommand{\Gal}{\mathrm{Gal}}

\newcommand{\Stab}{\mathrm{Stab}}
\newcommand{\Hol}{\mathrm{Hol}}

\newcommand{\Sym}{\operatorname{Sym}}

\newcommand{\GL}{\mathrm{GL}}

\newcommand{\ord}{\mathrm{ord}}

\newcommand{\End}{\operatorname{End}}

\newcommand{\Aut}{\operatorname{Aut}}

\newcommand{\Id}{\operatorname{Id}}
\newcommand{\Syl}{\operatorname{Syl}}
\newcommand{\wL} {{\widetilde{L}}}

\begin{document}

\large
\begin{center}
{\bf Computation of Hopf Galois structures on separable extensions and classification of those for degree twice an odd prime power}

\vspace{1cm}
Teresa Crespo and Marta Salguero

\vspace{0.3cm}
\footnotesize

Departament de Matem\`atiques i Inform\`atica, Universitat de Barcelona (UB), Gran Via de les
Corts Catalanes 585, E-08007 Barcelona, Spain, e-mail: teresa.crespo@ub.edu, msalguga11@alumnes.ub.edu

\end{center}

\date{\today}

\let\thefootnote\relax\footnotetext{Both authors acknowledge support by grant MTM2015-66716-P (MINECO/FEDER, UE).}

\normalsize
\begin{abstract} A Hopf Galois structure on a finite field extension $L/K$ is a pair $(H,\mu)$, where $H$ is a finite cocommutative $K$-Hopf algebra and $\mu$ a Hopf action. In this paper we present a program written in the computational algebra system Magma which gives all Hopf Galois structures on separable field extensions of a given degree and several properties of those. We show a table which summarizes the program results. Besides, for separable field extensions of degree $2p^n$, with $p$ an odd prime number, we prove that the occurrence of some type of Hopf Galois structure may either imply or exclude the occurrence of some other type. In particular, for separable field extensions of degree $2p^2$, we determine exactly the possible sets of  Hopf Galois structure types.

\end{abstract}

\section{Introduction}
A Hopf Galois structure on a finite extension of fields $L/K$ is a pair $(H,\mu)$, where $H$ is
a finite cocommutative $K$-Hopf algebra  and $\mu$ is a
Hopf action of $H$ on $L$, i.e a $K$-linear map $\mu: H \to
\End_K(L)$ giving $L$ a left $H$-module algebra structure and inducing a $K$-vector space isomorphism $L\otimes_K H\to\End_K(L)$.
Hopf Galois structures were introduced by Chase and Sweedler in \cite{C-S}.
For separable field extensions, Greither and
Pareigis \cite{G-P} give the following group-theoretic
equivalent condition to the existence of a Hopf Galois structure.

\begin{theorem}\label{G-P}
Let $L/K$ be a separable field extension of degree $g$, $\wL$ its Galois closure, $G=\Gal(\wL/K), G'=\Gal(\wL/L)$. Then there is a bijective correspondence
between the set of Hopf Galois structures on $L/K$ and the set of
regular subgroups $N$ of the symmetric group $S_g$ normalized by $\lambda (G)$, where
$\lambda:G \hookrightarrow S_g$ is the monomorphism given by the action of
$G$ on the left cosets $G/G'$.
\end{theorem}

For a given Hopf Galois structure on a separable field extension $L/K$ of degree $g$, we will refer to the isomorphism class of the corresponding group $N$ as the type of the Hopf Galois
structure. The Hopf algebra $H$ corresponding to a regular subgroup $N$ of $S_g$ normalized by $\lambda (G)$ is the Hopf subalgebra $\wL[N]^G$ of the group algebra $\wL[N]$ fixed under the action of $G$, where $G$ acts on $\wL$ by $K$-automorphisms and on $N$ by conjugation through $\lambda$. The Hopf action is induced by $n \mapsto n^{-1}(\overline{1})$, for $n \in N$, where we identify $S_g$ with the group of permutations of $G/G'$ and $\overline{1}$ denotes the class of $1_G$ in $G/G'$. It is known that the Hopf subalgebras of $\wL[N]^G$ are in 1-to-1 correspondence with the subgroups of $N$ stable under the action of $G$ (see e.g. \cite{CRV} Proposition 2.2) and that, given two regular subgroups $N_1, N_2$ of $S_g$ normalized by $\lambda (G)$, the Hopf algebras $\wL[N_1]^G$ and $\wL[N_2]^G$ are isomorphic if and only if the groups $N_1$ and $N_2$ are $G$-isomorphic.

Childs \cite{Ch1} gives an equivalent  condition to the existence of a Hopf Galois structure introducing the holomorph of the regular subgroup $N$ of $S_g$. We state the more precise formulation of this result due to Byott \cite{B} (see also \cite{Ch2} Theorem 7.3).

\begin{theorem}\label{theoB} Let $G$ be a finite group, $G'\subset G$ a subgroup and $\lambda:G\to \Sym(G/G')$ the morphism given by the action of
$G$ on the left cosets $G/G'$.
Let $N$ be a group of
order $[G:G']$ with identity element $e_N$. Then there is a
bijection between
$$
{\cal N}=\{\alpha:N\hookrightarrow \Sym(G/G') \mbox{ such that
}\alpha (N)\mbox{ is regular}\}
$$
and
$$
{\cal G}=\{\beta:G\hookrightarrow \Sym(N) \mbox{ such that }\beta
(G')\mbox{ is the stabilizer of } e_N\}
$$
Under this bijection, if $\alpha\in {\cal N}$ corresponds to
$\beta\in {\cal G}$, then $\alpha(N)$ is normalized by
$\lambda(G)$ if and only if $\beta(G)$ is contained in the
holomorph $\Hol(N)$ of $N$.
\end{theorem}

\noindent
{\bf Notation.} In the sequel, $L/K$ will denote a finite separable field extension, $\wL$ the normal closure of $L/K$, $G$ the Galois group $\Gal(\wL/K)$, $G'$ the Galois group $\Gal(\wL/L)$.

\vspace{0.2cm}

In \cite{CS1}, we presented a program written in the computational algebra system Magma which determines all Hopf Galois structures on a separable field extension of a given degree $g$. It was built on Theorem \ref{G-P} and was effective up to degree 11.
In this paper we present a program based on Theorem \ref{theoB} which determines all Hopf Galois structures on a separable field extension up to degree 31 and uses the Magma database of transitive groups which derives from the classification given in \cite{Hu}. We note that the bound on the degree is imposed by the fact that transitive groups are not classified beyond degree 31 and not by the effectiveness of our program. As for the one in \cite{CS1}, the program presented here distinguishes almost classically Galois structures, decides for the remaining ones if the Galois correspondence is bijective and classifies the Hopf Galois structures in Hopf algebra isomorphism classes. In Section \ref{algorithm} we give a description of the program and in Section \ref{table} we show a table which summarizes the obtained results. The Magma code and the output of the program for each degree can be found in \cite{CSw}.

In \cite{CS2} we studied Hopf Galois structures on a separable field extension of degree $p^n$, for $p$ an odd prime, $n\geq 2$. In this paper we consider separable field extensions of degree $2p^n$. In Section \ref{2pn}, we consider the general case $n\geq 2$. Theorems \ref{2pn1} and \ref{2pn2} concerns Hopf Galois structure types on a separable field extension of degree $2p^n$. The first one proves that the occurrence of cyclic type implies the occurrence of dihedral type and the second one that the occurrence of dihedral type excludes the occurrence of any type of exponent smaller that $2p^n$.

\vspace{0.2cm}
In Section \ref{section2p2} we study in more detail the case $n=2$. We describe the five groups of order $2p^2$ and determine the corresponding automorphism groups. In Theorem \ref{2p2} we obtain that if a separable field extension of degree $2p^2$ has a Hopf Galois structure whose type is one of the two groups of exponent $2p$ containing an element of order $2p$, then it has a Galois structure of type the only group of exponent $2p$ not containing an element of order $2p$. Finally Corollary \ref{sets} gives all possible sets of Hopf Galois structure types for  separable field extensions of degree $2p^2$.

\vspace{0.2cm}
We note that the results obtained have been intuited by performing an analysis of the outputs of our program.

\section{Description of the computation procedure}\label{algorithm}

Given a separable field extension $L/K$ of degree $g$, $\wL$ its Galois closure, $G=\Gal(\wL/K),$ \newline $G'=\Gal(\wL/L)$, the action of
$G$ on the left cosets $G/G'$ is transitive, hence the morphism $\lambda:G \rightarrow S_g$ identifies $G$ with a transitive subgroup of $S_g$, which is determined up to conjugacy. Moreover, if we enumerate the left cosets $G/G'$ starting with the one containing $1_G$, $\lambda(G')$ is equal to the stabilizer of $1$ in $\lambda(G)$. Therefore considering all separable field extensions $L/K$ of degree $g$ is equivalent to considering all transitive groups $G$ of degree $g$, up to conjugation. We note that these groups have been classified in \cite{Hu} up to $g=31$ and are included in the data base of the program Magma. We shall denote by $gTk$ the group given by Magma as \verb;TransitiveGroup(g,k);.

\vspace{0.2cm}
The program consists in an auxiliary function Automorphisms, a Main Function which determines the Hopf Galois structures and the additional functions giving the properties of these structures. We describe in more detail the two first ones since the others have already been described in \cite{CS1}.

\vspace{0.5cm}
\noindent
{\bf The function Automorphisms}

Given a pair of integers $(g,k)$ this function returns the group $\Aut(G)$ of automorphisms of the group $G=gTk$ and the group $\Aut(G,G')$ of automorphisms of $G$ sending $G'$ to itself. In order to obtain the latter, the function uses the permutation representation of $\Aut(G)$ to obtain a group $P$ of permutations isomorphic to  $\Aut(G)$. It then computes the set \verb;stabims; of images of $G'$ under $\Aut(G)$ and the action of the generators of $\Aut(G)$ on this set. This gives the embedding \verb;act; of $P$ into $\Sym(\verb;stabims;)$ and then the preimage of the stabilizer of 1 in $\Sym(\verb;stabims;)$ is a subgroup $Q$ of $P$ corresponding to $\Aut(G,G')$ by the permutation representation.

\vspace{0.5cm}
\noindent
{\bf The main function}
\begin{enumerate}
\item[Step 0] For each regular subgroup $N$ of $S_g$, we order it so that $n_j(1)=j$, for $1\leq j \leq g$, and compute its embedding $\lambda(N)$ in $S_g$ induced by the action on itself by left translation.
\item[Step 1] Given a transitive group $G$ of degree $g$ and a type of regular subgroups $N$ of $S_g$, we determine the subgroups $G^*$ of $H:=\Hol(N)$ which are isomorphic to $G$ and transitive and such that the stabilizer $\Stab(G^*,1)$ of $1$ in $G^*$ is isomorphic to the stabilizer $G'$ of $1$ in $G$.
\item[Step 2] For each $G^*$ obtained in Step 1, we look for an isomorphism from $G^*$ to $G$ sending $\Stab(G^*,1)$ to $G'$. Let $f$ be the isomorphism from $G^*$ to $G$ provided by Magma. If $|G|=g$, then $G'$ is trivial and $f$ will do. If this is not the case, we compare $f(\Stab(G^*,1))$ to the images of $G'$ by all automorphisms of $G$. If, for some $a \in \Aut G$, we have $f(\Stab(G^*,1))=a(G')$, then $h:=f \circ a^{-1}$ is the wanted isomorphism. Then $\beta=h^{-1}$ is the embedding $\beta$ as in Theorem \ref{theoB}.
\item[Step 3] We order $T:=G/G'$ so that $t_j(1)=j$, for $1\leq j \leq g$.
\item[Step 4] For each pair $(G^*,h)$ obtained in Step 2, we compute the whole set of isomorphisms $G^*$ to $G$ sending $\Stab(G^*,1)$ to $G'$ by composing $h$ with each element in $\Aut(G,G')$. We obtain then all $\beta$'s from $G$ to $\Hol(N)$ as in Theorem \ref{theoB}. For each such $\beta$ we determine the corresponding $\alpha(N)$ as in the proof of Theorem \ref{theoB}. We obtain then all regular subgroups of $S_g$ isomorphic to $N$ and normalized by $\lambda(G)$.
\end{enumerate}

We further determine those Hopf Galois structures for which the Galois correspondence is bijective and partition the set of Hopf Galois structures of a given type in Hopf algebra isomorphism classes with an optimized version of the algorithm presented in \cite{CS1} (see also \cite{CSw}). Previously we compute the embedding $\lambda(G)$ in $S_g$ induced by the action of $G$ by left translation on the set $T$ of cosets of $G$ modulo $G'$ accordingly with the ordering of $T$ in Step 3. Taking into account \cite{CS1} Proposition 6, we know that an almost classically Galois structure lies alone in its isomorphism class. Hence, we put these apart when performing the partition in Hopf algebra isomorphism classes. Furthermore we determine the Hopf Galois structures for which the Galois correspondence is bijective in a more effective way.

\section{Extensions of degree $2p^n$}\label{2pn}

If $N$ is a group of order $2p^n$, with $p$ an odd prime, its unique $p$-Sylow subgroup is either a cyclic group of order $p^n$ or a group of order $p^n$ and exponent $<p^n$. In the first case, $N$ is either a cyclic group or a dihedral group. In the second case, $N$ has exponent $<2p^n$.

\begin{theorem}\label{2pn1} Let $L/K$ be a separable extension of degree $2p^n$, $p$ an odd prime, $n\geq 1$.
If $L/K$ has a Hopf Galois structure of cyclic type, then it has a Hopf Galois structure of dihedral type.
\end{theorem}

\begin{proof}
 We shall see that $\Hol(D_{2p^n})$ contains a subgroup isomorphic to $C_{2p^n}$, acting regularly on $D_{2p^n}$ and such that its normalizer in $\Sym(D_{2p^n})$ is contained in $\Hol(D_{2p^n})$. The result will then follow from Theorem \ref{theoB}. Let us write $D_{2p^n}=\langle r,s \mid r^{p^n}=s^2=1, srs=r^{-1} \rangle$. The automorphism group of $D_{2p^n}$ is generated by the automorphisms $\varphi$ and $\psi$ defined as follows.

$$\begin{array}{cccl} \varphi:&s & \mapsto & rs \\ &r & \mapsto & r \end{array} \quad \begin{array}{cccl} \psi:&s & \mapsto & s \\ &r & \mapsto & r^i \end{array}$$

\noindent
where $i$ has order $p^{n-1}(p-1)$ modulo $p^n$ and $1 \leq i \leq p^n$.  The automorphisms $\varphi$ and $\psi$ have orders $p^n$ and $p^{n-1}(p-1)$, respectively and satisfy $\psi\varphi\psi^{-1}=\varphi^i$. We have then $\Hol(D_{2p^n})=\langle r,s,\varphi,\psi\rangle$, with $r,\varphi$ of order $p^n$ and commuting with each other; $s$ of order $2$ and $\psi$ of order $p^{n-1}(p-1)$ commuting with each other and with the further relations

$$srs=r^{-1}, \varphi s \varphi^{-1}=rs, \psi r \psi^{-1}=r^i, \psi\varphi\psi^{-1}=\varphi^i.$$

Now, the powers of the element $s\varphi$ satisfy

$$(s\varphi)^k = \left\{ \begin{array}{ll} r^{-k/2} \varphi^k & \text{ if \ } k \text{\ is even} \\sr^{-(k-1)/2} \varphi^k & \text{ if \ } k \text{\ is odd} \end{array} \right.$$

\noindent Hence $s\varphi$ has order $2p^n$ and $\langle s\varphi \rangle$ acts regularly on $D_{2p^n}$. Indeed, we have

$$(s\varphi)^k (Id) = \left\{ \begin{array}{ll} r^{-k/2}  & \text{ if \ } k \text{\ is even} \\sr^{-(k-1)/2} & \text{ if \ } k \text{\ is odd} \end{array} \right.$$

We want to see now that the normalizer of $\langle s\varphi \rangle$ in $\Sym(D_{2p^n})$ is included in $\Hol(D_{2p^n})$. We consider the element $h=\varphi^k \psi \in \Hol(D_{2p^n})$, where $k= -(i-1)/2$, if $i$ is odd, and $k=(i+p^n-1)/2$ if $i$ is even. We may check that $h$ has order $p^{n-1}(p-1)$ and satisfy $h(s\varphi)h^{-1}= r^k s \varphi^i \in \langle s\varphi \rangle$. Hence the normalizer of $\langle s\varphi \rangle$ in $\Hol(D_{2p^n})$ has order $2p^{2n-1} (p-1)$ and is then the whole normalizer of $\langle s\varphi \rangle$ in $\Sym(D_{2p^n})$.
\end{proof}

To prove the next theorem we need the following technical lemma.

\begin{lemma}\label{le} Let $n\geq 1$ be an integer number and $p$ be an odd prime number. Then $p+1$ has order $p^{n-1}$ modulo $p^n$. More precisely

$$(p+1)^{p^{n-2}} \equiv 1+p^{n-1} \pmod{p^n}, \quad (p+1)^{p^{n-1}} \equiv 1 \pmod{p^n}.$$

\end{lemma}

\begin{proof} We have $(p+1)^{p^{n-2}}=1+ p^{n-1} +\sum_{k=2}^{p^{n-2}} {p^{n-2} \choose k} p^k$. Now $v_p \left( {p^{n-2} \choose k}\right)=n-2-v_p(k)$ (see \cite{R}), hence $v_p \left( {p^{n-2} \choose k}p^k \right)=n-2+k-v_p(k) \geq n$ for $k\geq 2$. Now $(p+1)^{p^{n-1}} \equiv (1+p^{n-1})^p \equiv 1 \pmod{p^n}$.
\end{proof}

\begin{theorem}\label{2pn2} Let $L/K$ be a separable extension of degree $2p^n$, $p$ an odd prime, $n\geq 1$.
If $L/K$ has a Hopf Galois structure of dihedral type, then it has no Hopf Galois structure of type $N$, for any $N$ of exponent $<2p^n$.
\end{theorem}

\begin{proof}
If  $L/K$ has a Hopf Galois structure of dihedral type, then $G=\Gal(\wL/K)$ embeds into the holomorph $\Hol(D_{2p^n})$ as a transitive subgroup. We shall prove that every transitive subgroup of $\Hol(D_{2p^n})$ contains an element of order $p^n$. Let us write $D_{2p^n}=\langle r,s \rangle$ and $\Hol(D_{2p^n})=\langle r,s,\varphi,\psi\rangle$, as in the proof of Theorem \ref{2pn1}. We have $|\Aut D_{2p^n}|=p^{2n-1}(p-1), |\Hol(D_{2p^n}|=2p^{3n-1}(p-1)$. We write $H:=\Hol(D_{2p^n})$. A subgroup $G$ of $H$ is transitive if and only if $[G:\Stab_H(1)\cap G]=2p^n$. For a transitive subgroup $G$ of $H$, we have then $|G|=2p^ld$, with $n\leq l \leq 3n-1$, $d\mid p-1$. Let $\Syl(G)$ be a $p$-Sylow subgroup of $G$ (it is unique except for $p=3$, $d=2$). We have the equalities between indices

$$\begin{array}{ll} & [G:\Stab_H(1) \cap \Syl(G)]=[G:\Syl(G)][\Syl(G):\Stab_H(1)\cap \Syl(G)], \\
& [G:\Stab_H(1) \cap \Syl(G)]=[G:\Stab_H(1)\cap G][\Stab_H(1)\cap G:\Stab_H(1)\cap \Syl(G)].\end{array}$$

\noindent Taking into account that $[G:\Syl(G)]$ and $[\Stab_H(1)\cap G:\Stab_H(1)\cap \Syl(G)]$ are prime to $p$ and $[\Syl(G):\Stab_H(1)\cap \Syl(G)]$ is a $p$-power, we obtain

$$G \text{\ transitive \ } \Rightarrow [\Syl(G):\Stab_H(1)\cap \Syl(G)]=p^n.$$

\noindent Let us determine now $\Syl(H)$. We have $|\Syl(H)|=p^{3n-1}$. The elements $r, \varphi, \psi^{p-1}$ belong to $\Syl(H)$. We write $\chi:=\psi^{p-1}$. Now $\langle \varphi, \chi \rangle = \langle \varphi \rangle \rtimes \langle \chi \rangle$ has order $p^{2n-1}$ and $\langle r, \varphi, \chi \rangle = \langle r \rangle \rtimes \langle \varphi ,\chi \rangle$ has order $p^{3n-1}$, hence $\Syl(H)=\langle r, \varphi, \chi \rangle$. Therefore $\Syl(H)\cap \Stab_H(1)=\langle \varphi, \chi \rangle$ has order $p^{2n-1}$. By Lemma \ref{le}, $p+1$ has order $p^{n-1}$ modulo $p^n$, hence we may assume $\chi \varphi \chi^{-1}=\varphi^{p+1}$. We will now characterize the elements of order $p^n$ in $H$. We determine first the elements of order $p^n$ in $\Syl(\Aut D_{2p^n})=\langle \varphi,\chi\rangle$.
We shall prove

\begin{equation}\label{eq1}
\varphi^i \chi^j, 0\leq i \leq p^{n}-1, 0\leq j \leq p^{n-1}-1 \text{\  has order \ }p^n \Leftrightarrow \varphi^i \text{ \  has order \ }p^n
\end{equation}

\noindent By induction, we prove $(\varphi^i \chi^j)^k=\varphi^{i(\sum_{l=0}^{k-1} (p+1)^{lj})}\chi^{kj}$. Since $\chi$ has order $p^{n-1}$, we have in particular $(\varphi^i \chi^j)^{p^{n-1}}=\varphi^{i(\sum_{l=0}^{p^{n-1}-1} (p+1)^{lj})}$. Letting $j=p^a j_0$, with $p\nmid j_0$, we have

$$\sum_{l=0}^{p^{n-1}-1} (p+1)^{lj}=\dfrac{1-(p+1)^{p^{n-1}j}}{1-(p+1)^j}=\dfrac{1-(p+1)^{p^{n+a-1}j_0}}{1-(p+1)^{p^aj_0}}.$$

\noindent By Lemma \ref{le}, we have $(p+1)^{p^{n+a-1}}=1+p^{n+a}+\lambda p^{n+a+1}$, for some integer $\lambda$ and $(p+1)^{p^a}=1+p^{a+1}+\mu p^{a+2}$, for some integer $\mu$, which implies $(p+1)^{p^{n+a-1}j_0}=1+j_0 p^{n+a}+\lambda' p^{n+a+1}$, for some integer $\lambda'$ and $(p+1)^{p^aj_0}=1+j_0p^{a+1}+\mu' p^{a+2}$, for some integer $\mu'$. We obtain then

\begin{equation}\label{eq2}
\sum_{l=0}^{p^{n-1}-1} (p+1)^{lj} \equiv p^{n-1} \pmod{p^n},
\end{equation}

\noindent hence

$$\varphi^{i(\sum_{l=0}^{p^{n-1}-1} (p+1)^{lj})}= 1 \Leftrightarrow p\mid i \Leftrightarrow \varphi^i \text{ \  has order \ } \leq p^n.$$

\noindent We have then proved (\ref{eq1}). Now, for $f \in \Aut(D_{2p^n})$, we have

$$(r^m,f)^k=(r^mf(r^m)\dots f^{k-1}(r^m),f^k),$$

\noindent hence, if $f$ has order $p^n$, then $(r,f)$ has order $p^n$. Now, if $f=\varphi^{pi} \chi^j$, then $(r^m,f)^{p^{n-1}}=(r^{m(\sum_{l=0}^{p^{n-1}-1} (p+1)^{lj})},\Id)$. By (\ref{eq2}), we obtain that $(r^m,\varphi^{pi} \chi^j)$ has order $p^n$ if and only if $r^m$ has order $p^n$ if and only if $p\nmid m$. Together we have obtained $\ord(r^m,\varphi^i \chi^j)=p^n$ if and only if $p\nmid m$ or $p\nmid i$. From this characterization of the elements in $H$ of order $p^n$, we obtain that the subset $F$ of elements of order $<p^n$ in $H$ is a subgroup of $H$ and has order $p^{3n-3}$. Moreover $|F\cap \Stab_H(1)|=p^{2n-2}$, hence $[F:F\cap \Stab_H(1)]=p^{n-1}$.

Let now $G$ be a subgroup of $H$ with no element of order $p^n$. We have then $\Syl(G) \subset F$ and $[\Syl(G):\Stab_H(1)\cap \Syl(G)]\leq [F:F\cap \Stab_H(1)]=p^{n-1}$, hence $G$ is not transitive. We have then proved that every transitive subgroup of $\Hol(D_{2p^n})$ contains an element of order $p^n$.

Let now $N$ be a group of order $2p^n$ and exponent $<2p^n$. We shall prove that $\Hol(N)$ has no elements of order $p^n$. Let $P$ be the unique $p$-Sylow subgroup of the group $N$. We have then $N=P\rtimes C_2$ and $P$ is a non-cyclic group of order $p^n$. If $\varphi \in \Aut N$, then $\varphi(P)=P$. By \cite{Ko}, Corollary 4.3, $\varphi_{|P}^{p^{n-1}}=\Id_P$. For $a$ an element of order 2 in $N$, we have $\varphi(a)=ax$, for some $x \in P$ satisfying $axa=x^{-1}$. In the particular case where $N$ is the direct product of $P$ and $C_2$, the only possibility is $\varphi(a)=a$. Now $\varphi(a)=ax$ implies

$$\varphi^{p^{n-1}}(a)=ax\varphi(x)\varphi^2(x) \dots\varphi^{p^{n-1}-1}(x)$$

\noindent
and by \cite{Ko}, Theorem 4.4, we have $x\varphi(x)\varphi^2(x) \dots\varphi^{p^{n-1}-1}(x)=e$, hence $\varphi^{p^{n-1}}(a)=a$. Since an element $\varphi$ in $\Aut(N)$ is determined by $\varphi_{|P}$ and $\varphi(a)$, we have obtained $\varphi^{p^{n-1}}=\Id_N$ for all $\varphi \in \Aut(N)$. Now, since $|\Syl_p(\Hol(N))|=|\Syl_p(N)||\Syl_p(\Aut(N))|$, we have $\Syl_p(\Hol(N))=P\rtimes \Aut(N)$. For $(y,\varphi) \in P\rtimes \Aut(N)$, we have

$$(y,\varphi)^{p^{n-1}}=(y\varphi(y)\varphi^2(y) \dots\varphi^{p^{n-1}-1}(y),\varphi^{p^{n-1}}).$$

\noindent By \cite{Ko}, Theorem 4.4, we have $y\varphi(y)\varphi^2(y) \dots\varphi^{p^{n-1}-1}(y)=e$ and we have proved above that $\varphi^{p^{n-1}}=\Id_N$. We have then that $\Hol(N)$ has no elements of order $p^n$.

In conclusion, if  $L/K$ has a Hopf Galois structure of dihedral type, we have proved that $G=\Gal(\wL/K)$ contains an element of order $p^n$. Since for $N$ a group of order $2p^n$ and exponent $<2p^n$, $\Hol(N)$ has no elements of order $p^n$, $G$ cannot be a subgroup of $\Hol(N)$, hence $L/K$ has no Hopf Galois structure of type $N$.
\end{proof}

\section{Extensions of degree $2p^2$}\label{section2p2}

Let $p$ denote an odd prime. If $G$ is a group of order $2p^2$, the $p$-Sylow subgroup $Syl_p$ of $G$ is normal in $G$. Hence $G$ is the direct or semidirect product of $Syl_p$ and $C_2$. The abelian groups of order $2p^2$ are $C_{p^2} \times C_2 \simeq C_{2p^2}$ and $C_p \times C_p \times C_2 \simeq C_p \times C_{2p}$.

\begin{enumerate}[1)]
\item If $Syl_p=C_{p^2}$ and $G$ is not abelian, since $\Aut C_{p^2} \simeq (\Z/p^2 \Z)^*$ has a unique element of order 2, there is a unique semidirect product and $G$ is isomorphic to the dihedral group $D_{2p^2}$.
\item If $Syl_p=C_p\times C_p=\langle a \rangle \times \langle b \rangle$ and $G$ is not abelian, since $\Aut (C_p\times C_p) \simeq \GL(2,p)$ has three elements of order 2, up to conjugation, namely $\left(\begin{smallmatrix} -1&0\\0&1 \end{smallmatrix} \right), \left(\begin{smallmatrix} 1&0\\0&-1 \end{smallmatrix} \right), \left(\begin{smallmatrix} -1&0\\0&-1 \end{smallmatrix} \right)$, there are three possible actions of $C_2$ on $Syl_p$:

    $$ \begin{array}{ccc} a & \mapsto & a^{-1} \\ b & \mapsto & b \end{array} , \quad \begin{array}{ccc} a & \mapsto & a \\ b & \mapsto & b^{-1} \end{array}, \quad \begin{array}{ccc} a & \mapsto & a^{-1} \\ b & \mapsto & b^{-1} \end{array}.$$

    \noindent The two first ones give both clearly groups isomorphic to $D_{2p} \times C_p$. The third one gives a group which we denote by $(C_p\times C_p) \rtimes C_2$. This group is not isomorphic to $D_{2p} \times C_p$ since the center of $D_{2p} \times C_p$ has order $p$ whereas the center of $(C_p\times C_p) \rtimes C_2$ is trivial.
\end{enumerate}

There are then 5 groups of order $2p^2$, up to isomorphism. We determine now the automorphism group for each of them.

\begin{enumerate}[1)]
\item $\Aut C_{2p^2} \simeq (\Z/2p^2 \Z)^*$ is cyclic of order $p(p-1)$.
\item For $G=C_p\times C_p \times C_2=\langle a \rangle \times \langle b \rangle \times \langle c \rangle$, $c$ is the unique element of order 2. An element in $\Aut G$ is then given by $a \mapsto a^i b^j, b\mapsto a^k b^l, c \mapsto c$, with $0\leq i,j,k,l \leq p-1, p\nmid il-jk$. We have then $\Aut G \simeq \GL(2,p)$ and $|\Aut G|=(p^2-1)(p^2-p)=p(p+1)(p-1)^2$.
\item For $G=D_{2p} \times C_p$, we write $D_{2p}=\langle r,s|r^p=s^2=1,srs=r^{-1} \rangle$ and denote by $c$ a generator of $C_p$. An automorphism of $G$ is given by $r \mapsto r^i, s \mapsto r^j s, c \mapsto c^k$, with $1\leq i,k\leq p-1, 0\leq j \leq p-1$. We have then $|\Aut G|=p(p-1)^2$.
\item For $G=(C_p\times C_p) \rtimes C_2$, we write $C_p\times C_p=\langle a \rangle \times \langle b \rangle$ and $C_2=\langle c \rangle$. An automorphism of $G$ is given by $a \mapsto a^ib^j, b \mapsto a^k b^l, c \mapsto a^mb^n c$, with $0\leq i,j,k,l,m,n\leq p-1, p\nmid il-jk$. We have then $|\Aut G|=(p^2-1)(p^2-p)p^2=p^3(p+1)(p-1)^2$.
\item For $G=D_{2p^2}=\langle r,s|r^{p^2}=s^2=1,srs=r^{-1} \rangle$, an automorphism is given by $r\mapsto r^i, s \mapsto r^j s$, with $0\leq i,j\leq p^2-1, p\nmid i$. We have then $|\Aut G|=(p^2-p)p^2=p^3(p-1)$.
\end{enumerate}

\begin{theorem}\label{2p2} Let $L/K$ be a separable extension of degree $2p^2$, $p$ an odd prime. If $L/K$ has a Hopf Galois structure of type $C_p\times C_{2p}$  or $C_p \times D_{2p}$, then it has a Hopf Galois structure of type $(C_p\times C_p)\rtimes C_2$.
\end{theorem}

\begin{proof} Let us assume that $L/K$ has a Hopf Galois structure of type $C_p\times C_{2p}$. Then $G=\Gal(\widetilde{L}/K)$ is isomorphic to a transitive subgroup of $\Hol(C_p\times C_{2p})$.  We shall see that $\Hol(C_p\times C_p)\rtimes C_2)$ contains a subgroup isomorphic to $C_p\times C_{2p}$, acting regularly on $(C_p\times C_p)\rtimes C_2$ and such that its normalizer in $\Sym((C_p\times C_p)\rtimes C_2)$ is contained in $\Hol((C_p\times C_p)\rtimes C_2)$. Let us write $(C_p\times C_p)\rtimes C_2=(\langle a \rangle \times \langle b \rangle) \rtimes \langle c \rangle$. An automorphism $\varphi$ of $(C_p\times C_p)\rtimes C_2$ is given by

$$\begin{array}{cccl} \varphi:& a & \mapsto & a^i b^j \\& b & \mapsto & a^k b^l \\& c & \mapsto & a^m b^n c \end{array} \quad p \nmid il-jk.$$

\noindent Let us consider the elements $(a,\Id), (b,\Id), (c,\varphi_1)$ in $\Hol((C_p\times C_p)\rtimes C_2)$, where $\varphi_1$ is defined by $\varphi_1(a)=a^{-1},
\varphi_1(b)=b^{-1}, \varphi_1(c)=c$. We may check that $(a,\Id)$ and $(b,\Id)$ have order $p$, $(c,\varphi_1)$ has order 2 and all three elements commute with each other, hence generate a subgroup $F_1$ isomorphic to $C_p\times C_{2p}$. We may see that the orbit of $1$ under the action of $F_1$ contains $(C_p\times C_p)\rtimes C_2$. Indeed

$$1 \stackrel{(a,\Id)^i}{\longmapsto}a^i \stackrel{(b,\Id)^j}{\longmapsto}a^i b^j \stackrel{(c,\varphi_1)}{\longmapsto}a^{-i} b^{-j} c.$$

\noindent By computation we obtain that the elements in $\Hol((C_p\times C_p)\rtimes C_2)$ which normalize $F_1$ are precisely those of the form $(x,\psi_1)$, where $x$ is any element in $(C_p\times C_p)\rtimes C_2$ and $\psi_1$ is defined by

$$\begin{array}{cccl} \psi_1:& a & \mapsto & a^i b^j \\& b & \mapsto & a^k b^l \\& c & \mapsto & c \end{array} \quad p \nmid il-jk.$$

We obtain then than the normalizer of $F_1$ in  $\Hol((C_p\times C_p)\rtimes C_2)$ has order $2p^2(p^2-1)(p^2-p)$ and is then the whole normalizer of $F_1$ in $\Sym((C_p\times C_p)\rtimes C_2)$. Theorem \ref{theoB} gives then  the wanted result.

Let us assume now that $L/K$ has a Hopf Galois structure of type $C_p\times D_{2p}$. Then $G=\Gal(\widetilde{L}/K)$ is isomorphic to a transitive subgroup of $\Hol(C_p\times D_{2p})$.  We shall see that $\Hol(C_p\times C_p)\rtimes C_2)$ contains a subgroup isomorphic to $C_p\times D_{2p}$, acting regularly on $(C_p\times C_p)\rtimes C_2$ and such that its normalizer in $\Sym((C_p\times C_p)\rtimes C_2)$ is contained in $\Hol((C_p\times C_p)\rtimes C_2)$.

Let us consider the elements $(a,\Id), (b,\Id), (c,\varphi_2)$ in $\Hol((C_p\times C_p)\rtimes C_2)$, where $\varphi_2$ is defined by $\varphi_2(a)=a,
\varphi_2(b)=b^{-1}, \varphi_2(c)=c$. We may check that $(c,\varphi_2)$ has order 2, commutes with $(b,\Id)$ and satisfies $(c,\varphi_2)(a,\Id)(c,\varphi_2)=(a,\Id)^{-1}$. The three elements generate then a subgroup $F_2$ isomorphic to $C_p\times D_{2p}$. We may see that the orbit of $1$ under the action of $F_2$ contains $(C_p\times C_p)\rtimes C_2)$. Indeed

$$1 \stackrel{(a,\Id)^i}{\longmapsto}a^i \stackrel{(b,\Id)^j}{\longmapsto}a^i b^j \stackrel{(c,\varphi_2)}{\longmapsto}a^{-i} b^{j} c.$$

\noindent By computation we obtain that the elements in $\Hol((C_p\times C_p)\rtimes C_2)$ which normalize $F_2$ are precisely those of the form $(x,\psi_2)$, where $x$ is any element in $(C_p\times C_p)\rtimes C_2$ and $\psi_2$ is defined by

$$\begin{array}{cccl} \psi_2:& a & \mapsto & a^i  \\& b & \mapsto & b^l \\& c & \mapsto & a^mc \end{array} \quad p \nmid il.$$

We obtain then than the normalizer of $F_2$ in  $\Hol((C_p\times C_p)\rtimes C_2)$ has order $2p^3(p-1)^2$ and is then the whole normalizer of $F_2$ in $\Sym((C_p\times C_p)\rtimes C_2)$. Again Theorem \ref{theoB} gives the wanted result.
\end{proof}

\begin{corollary}\label{sets} Let $L/K$ be a separable extension of degree $2p^2$, $p$ an odd prime. Then either $L/K$ has no Hopf Galois structures or the set of types of Hopf Galois structures on $L/K$ is one of the following:

$$\begin{array}{c} \{ D_{2p^2} \}, \quad \{ (C_p\times C_p) \rtimes C_2 \}, \quad \{ D_{2p^2}, C_{2p^2} \}, \quad \{ (C_p\times C_p) \rtimes C_2, C_p\times C_{2p} \}, \\ \{ (C_p\times C_p) \rtimes C_2,C_p \times D_{2p} \}, \quad  \{ (C_p\times C_p) \rtimes C_2, C_p\times C_{2p}, C_p \times D_{2p} \}.\end{array}$$

\end{corollary}

\begin{proof} By Theorems \ref{2pn1} and \ref{2p2} the possibilities not considered in the corollary do not occur. Let us see now that all cases listed do occur by exhibiting an example. We denote again by $\wL$ a Galois closure of $L/K$ and $G=\Gal(\wL/K)$. For $G=S_{2p^2}$, the whole symmetric group in $2p^2$ letters, $L/K$ has no Hopf Galois structures, by \cite{G-P} Corollary 4.8. If for some group $N$ of order $2p^2$, we have either $G=N$ or $G=\Hol(N)$, then $L/K$ has a Hopf Galois structure of type $N$. If for some group $N$ of order $2p^2$, we have $|G|>|\Hol(N)|$, then $L/K$ has no Hopf Galois structures of type $N$ by Theorem \ref{theoB}. These facts, together with Theorems \ref{2pn1} and \ref{2p2}, allows us to obtain the results listed in the following table giving the types $N$ occurring for each $G$ in the first column, except those marked with a question mark.

\vspace{0.5cm}
\begin{center}
\begin{tabular}{|l||c|c|c|c|c|}
\hline
Galois group $\diagdown$ types &  $C_{2p^2}$& $D_{2p^2}$ & $C_p\times C_{2p}$ & $C_p \times D_{2p}$ & $(C_p\times C_p) \rtimes C_2$ \\
\hline
\hline
$G=\Hol(D_{2p^2})$ & No  & Yes & No& No & No \\
\hline
$G=\Hol((C_p\times C_p) \rtimes C_2)$ &  No  & No & No  & No   & Yes \\
\hline
$G=C_{2p^2}$ & Yes & Yes& No & No & No \\
\hline
$G=\Hol(C_p\times C_{2p})$ & No & No & Yes & No  & Yes \\
\hline
$G=\Hol(C_p\times D_{2p})$ & No & No & No ? & Yes & Yes \\
\hline
$G=C_p\times C_{2p}$ & No & No & Yes & Yes ? & Yes \\
\hline
\end{tabular}
\end{center}

\vspace{0.5cm}
We now resolve the question marks.

If $G=\Hol(C_p\times D_{2p})$, let us see that it does not have structures of type  $N:=C_p\times C_{2p}$. Equivalently, we shall prove that $G$ do not embed into $\Hol(N)$. To this end, we observe that the $p$-Sylow groups $\Syl(G)$ and $\Syl(\Hol(N))$ have both order $p^3$. We shall prove that they are not isomorphic. Let us write $N=C_p\times C_p \times C_2=\langle a \rangle \times \langle b \rangle \times \langle c \rangle$, as above. We consider the following elements in $\Hol(N)$: $A=(a,\Id), B=(b^{-1},\Id), C=(1,\varphi)$, where $\varphi$ is defined by $\varphi(a)=ab, \varphi(b)=b, \varphi(c)=c$. These three elements have order $p$ and satisfy $AB=BA, BC=CB, AC=BCA$, hence generate a subgroup of $\Hol(N)$ isomorphic to the Heisenberg group $H_p$ (see \cite{CS2} 3.2). We have then $\Syl(\Hol(N)) \simeq H_p$. Let us write now $D_{2p}\times C_p=\langle r,s,c \rangle$, as above. We consider the elements $(r,\Id), (c,\Id), (1,\psi)$ in $G$, where $\psi$ is defined by $\psi(r)=r,\psi(s)=rs, \psi(c)=c$. These elements have order $p$ and commute with each other, hence generate a subgroup of $G$ isomorphic to $C_p\times C_p \times C_p$. We have then $\Syl(G) \simeq C_p\times C_p \times C_p$.

Finally if $G=C_p\times C_{2p}$, we shall see that $L/K$ has Hopf Galois structures of type $C_p\times D_{2p}$. To this end, we prove that $C_p\times C_{2p}$ embeds regularly in $\Hol(C_p\times D_{2p})$. We write $D_{2p}=\langle r,s|r^p=s^2=1,srs=r^{-1} \rangle$ and denote by $c$ a generator of $C_p$. The elements $(r,\Id)$ and $(c,\Id)$ have order $p$ and commute with each other. The element $(s,\varphi_0)$, where $\varphi_0$ is defined by $\varphi_0(r)=r^{-1}, \varphi_0(s)=s, \varphi_0(c)=c$ has order 2 and commutes with $(r,\Id)$ and $(c,\Id)$, hence the three elements together generate a subgroup isomorphic to $C_p\times C_{2p}$. Moreover, this subgroup is regular since the orbit of $1$ under its action is the whole $(C_p\times D_{2p})$. Indeed, we have

$$1 \stackrel{(r,\Id)^i}{\longmapsto}r^i \stackrel{(c,\Id)^j}{\longmapsto}r^i c^j \stackrel{(s,\varphi_0)}{\longmapsto}sr^{-i} c^{j}=r^i s c^j.$$

\end{proof}

\section{Summary of the program output}\label{table}
Table \ref{fig} is a compendium of the computation results. In it we give for every degree $g$ the total number of transitive groups of degree $g$ and the number Max of transitive groups of degree $g$ whose order does not exceed the order of the holomorphs of all the groups of order $g$; the number of possible types of Hopf Galois structures; the total number of Hopf Galois structures and the number of the almost classically Galois ones; the number of Hopf Galois structures with bijective Galois correspondence and the number of those which are not almost classically Galois; the number of Hopf algebra isomorphism classes in which the Hopf Galois structures are partitioned (which correspond to $G$-isomorphism classes of the corresponding regular groups $N$) and the number of those for Galois extensions (i.e. when $G'=\Gal(\wL/L)$ is trivial); and finally the execution times in seconds and the memory used in megabytes (except for $g=16$ which could not be computed at once).

\newpage
\begin{landscape}
\begin{table}[ht]\label{fig}
\centering
\caption{Summary of results}
\label{fig}
\vspace{0.2 cm}
\begin{tabular}{|c||c|c||c||c|c||c|c||c|c||c||c|}
\hline
\multicolumn{1}{|c||}{\bf Degree} & \multicolumn{2}{|c||}{\bf Transitive Groups} &\multicolumn{1}{|c||}{\bf Types} &\multicolumn{2}{|c||}{\bf HG struct.} &\multicolumn{2}{|c||}{\bf BC} &\multicolumn{2}{|c||}{\bf $G$-iso} & \multicolumn{1}{|c||}{\bf Execution} &\multicolumn{1}{|c|}{\bf Memory} \\
\cline{2-3} \cline{5-10}
 \multicolumn{1}{|c||}{} &  \multicolumn{1}{|c|}{\quad Total\quad} & \multicolumn{1}{|c||}{Max} & \multicolumn{1}{|c||}{} & \multicolumn{1}{|c|}{Total} & \multicolumn{1}{|c||}{a-c}& \multicolumn{1}{|c|}{Total} & \multicolumn{1}{|c||}{not a-c}& \multicolumn{1}{|c|}{Total}& \multicolumn{1}{|c||}{Galois}&\multicolumn{1}{|c||}{{\bf time} (s)}&\multicolumn{1}{|c|}{{\bf used} (MB)}\\ \hline \hline
12&301&129&5&249&56&81&25&165&48&$\approx$ 18& $\approx$ 31 \\ \hline
13&9&6&1&6&6&6&0&6&1&$\approx$ 1 & $\approx$ 18 \\ \hline
14&63&25&2&32&14&19&5&26&6&$\approx$ 2 & $\approx$ 22 \\ \hline
15&104&11&1& 8&8&8&0&8&1&$\approx$ 1 &$\approx$ 22\\ \hline
16&1954&1906&14&49913&2636&9331&6695&26769&6717& -- & -- \\ \hline
17& 10&5&1&5&5&5&0&5&1&$\approx$ 1 &$\approx$ 30\\ \hline
18&983&528& 5&881&123&253&130&525&79&$\approx$ 206&$\approx$ 113\\ \hline
19& 8&6&1&6&6&6&0&6&1&$\approx$ 1& $\approx$ 26\\ \hline
20&1117&170&5&434&79&156&77&266&55&$\approx$ 57&$\approx$ 48\\ \hline
21& 164&26&2&78&22&46&24&42&8&$\approx$ 5& $\approx$ 26\\ \hline
22& 59&18&2&36&14&19&5&26&6&$\approx$ 5& $\approx$ 26\\ \hline
23& 7&4&1&4&4&4&0&4&1&$\approx$ 1 & $\approx$ 18\\ \hline
24& 25000&9738&15&14908&844&2682&1838&8353&1896&$\approx$ 9730& $\approx$ 1327\\ \hline
25& 211&90&2&106&70&74&4&82&12&$\approx$ 32& $\approx$ 175\\ \hline
26& 96&37&2&58&22&35&13&46&6&$\approx$ 12& $\approx$ 27\\ \hline
27& 2392&1547&5&6699&766&1100&334&2030&547&$\approx$ 5375& $\approx$ 500\\ \hline
28& 1854&214&4&388&84&143&59&256&40&$\approx$ 63& $\approx$ 33\\ \hline
29& 8&6&1&6&6&6&0&6&1&$\approx$ 1& $\approx$ 22\\ \hline
30& 5712&483&4&479&99&197&98&373&36&$\approx$ 113& $\approx$ 40 \\ \hline
31& 12&8&1&8&8&8&0&8&1&$\approx$ 1& $\approx$ 22\\ \hline
\end{tabular}
\end{table}
\end{landscape}
\newpage

\section*{Acknowledgments.} We are very thankful to Professor Derek Holt for kindly sending to us the Magma code of the function Automorphisms.

\end{document}